\documentclass[12 pt]{amsart}
\usepackage{amsmath,times,epsfig,amssymb,amsbsy,amscd,amsfonts,amstext,color,bm}
\usepackage[arrow,matrix]{xy}

\usepackage{tikz}
\usetikzlibrary{intersections,calc,arrows.meta}

\numberwithin{equation}{section}

\theoremstyle{plain}
\newtheorem{theorem}{Theorem}[section]
\newtheorem{corollary}[theorem]{Corollary}

\newtheorem{conjecture}[theorem]{Conjecture}

\theoremstyle{definition}
\newtheorem{definition}[theorem]{Definition}
\newtheorem{example}[theorem]{Example}

\theoremstyle{remark}
\newtheorem{remark}[theorem]{Remark}

\newcommand{\A}{\mathcal{A}}
\newcommand{\Z}{\mathbb{Z}}

\newcommand{\C}{\mathbb{C}}

\newcommand{\LL}{\mathcal{L}}

\newcommand{\Aut}{\operatorname{Aut}}
\newcommand{\rank}{\operatorname{rank}}

\newcommand{\Ext}{\operatorname{Ext}}
\newcommand{\Hom}{\operatorname{Hom}}
\newcommand{\Ker}{\operatorname{Ker}}

\newcommand{\DP}{\mathcal{A_{DS}}}
\newcommand{\decID}{\mathcal{A_{ID}}}

\begin{document}

%\title[]{}
\title[Double covering]{Betti numbers and torsions in homology groups of double coverings}

\begin{abstract}
Papadima and Suciu proved an inequality between the ranks of 
the cohomology groups of the Aomoto complex with finite field coefficients 
and the twisted cohomology groups, and conjectured that 
they are actually equal for certain cases associated with the Milnor fiber 
of the arrangement. Recently, 
an arrangement (the icosidodecahedral arrangement) with the following 
two peculiar properties was found: (i) the strict version of 
Papadima-Suciu's inequality holds, and 
(ii) the first integral homology of the Milnor fiber has a non-trivial $2$-torsion. 
In this paper, we investigate the 
relationship between these two properties for double covering spaces. 
We prove that (i) and (ii) are actually equivalent. 
\end{abstract}

\author{Suguru Ishibashi}
\address{Suguru Ishibashi, ARISE analytics}
\email{sgr.ishibashi@gmail.com}

\author{Sakumi Sugawara}
\address{Sakumi Sugawara, Hokkaido University}
\email{sugawara.sakumi.f5@elms.hokudai.ac.jp}

\author{Masahiko Yoshinaga}
\address{Masahiko Yoshinaga, Osaka University}
\email{yoshinaga@math.sci.osaka-u.ac.jp}

%\thanks{The author thanks 
%This work was partially supported by 
%JSPS KAKENHI Grant Numbers JP18H01115, JP15KK0144}

%Primary 14C21,
%14F99, 32S22 ; Secondary 14E05, 14H50.}

\date{\today}
\maketitle

\tableofcontents

\section{Introduction}

\label{sec:intro}

Double coverings of CW-complexes are 
well-studied subject in topology. We first recall several classical 
constructions related to double coverings \cite{hau}. Let $X$ be a 
connected CW-complex. Let $p:Y\longrightarrow X$ be an unbranched 
double covering with connected $Y$. Fix a base point $x_0\in X$. 
By definition, the fiber $p^{-1}(x_0)$ consists of two points. Let 
$\gamma:[0, 1]\longrightarrow X$ be a closed path with 
$\gamma(0)=\gamma(1)=x_0$. The local triviality of the covering 
enables us to construct parallel transport of the fiber along $\gamma$ and 
a bijection $p^{-1}(x_0)\longrightarrow p^{-1}(x_0)$. 
Because of the homotopy lifting property, this bijection depends only on 
the homotopy type of $\gamma$. Hence, the group homomorphism 
$\pi_1(X, x_0)\longrightarrow\Aut(p^{-1}(x_0))$ is well-defined, 
and is called the characteristic map, or the monodromy, 
of the double covering. 
Since $\Aut(p^{-1}(x_0))\simeq\{\pm 1\}\simeq \Z_2$, the double covering 
assigns an element of $\Z_2$ to each closed path. Thus an element 
of $H^1(X, \Z_2)$ is attached to a double covering, which is called 
the characteristic class of the double covering $p:Y\longrightarrow X$. 
Conversely, any nonzero element $\omega\in H^1(X, \Z_2)$ determines 
a double covering $p:X^\omega\longrightarrow X$. 

Although the double covering $p:X^\omega\longrightarrow X$ is 
determined by $\omega\in H^1(X, \Z_2)$, the topology of $X^\omega$ is 
not so simple. 
In fact, double coverings have attracted a lot of attention recently 
in the topological study of hyperplane arrangements 
\cite{liu-liu, suc-boc, sug, y-dou}. 

%Actually, recently, double coverings attracts a lot of attention in 
%the topology of hyperplane arrangement 

A hyperplane arrangement is a finite collection of hyperplanes in 
a linear (affine, or projective) space. 
The relationship between the topological and combinatorial structures 
of arrangements has been much studied. 
%Lots of studies focusing on 
%relationships between topology and combinatorial structures 
%have been made 
\cite{ds-multinets, dim, fal-yuz, gue-sos, gue-lat, lib, ot, PS-modular, suc-fundam, suc-mil, suc-boc, sug, y-mil, y-dou}. 

Hyperplane arrangements are also important 
as hypersurfaces with non-isolated singularities. 
The topology of the Milnor fiber of a central hyperplane arrangement is 
one of the central topics in the theory of arrangements 
\cite{cds, cs-mil, cs-char, ds-multinets, suc-mil, y-mil}. 
However, even the first Betti number of the Milnor fiber of hyperplane 
arrangements have not yet been understood well. 

The basic strategy to study the Milnor fiber $F_\A$ of a 
(central) hyperplane arrangement $\A$ is to use the monodromy 
action and eigenspace decomposition 
\begin{equation}
\label{eq:eigen}
H^k(F, \C)\simeq\bigoplus H^k(F, \C)_\lambda, 
\end{equation}
where $\lambda$ runs complex 
numbers satisfying $\lambda^{|\A|}=1$ and 
$H^k(F, \C)_\lambda$ is 
the $\lambda$-eigenspace of the monodromy action. 

Papadima and Suciu studied the relationship between the monodromy 
eigenspace and the so-called Aomoto complex with finite field 
coefficients \cite{ps-sp, PS-modular}. They proved an inequality 
between the dimension $\dim_\C H^k(F, \C)_\lambda$ and the 
cohomology of the Aomoto complex (see (\ref{psineq}) below). 
Furthermore, they conjectured that the equality holds, which enables us 
to formulate a combinatorial 
procedure computing $\dim_\C H^k(F, \C)_\lambda$ purely combinatorial 
way using Aomoto complex. 

However, recently, \cite{y-dou} exhibits an example 
of an arrangement of $16$ planes in $\C^3$ 
(the icosidodecahedral arrangement $\decID$, see Example \ref{ex:id} 
for details) 
for which 
\begin{itemize}
\item[(i)] Papadima-Suciu's conjectural equality breaks and 
the strict inequality holds. 
%$\dim H^1(F, \C)_{-1}$ of the $(-1)$-eigenspace, and 
\item[(ii)] the integral homology of the Milnor fiber $H_1(F, \Z)$ has 
a non-trivial $2$-torsion. 
\end{itemize}
(Examples of torsion in the first homology of Milnor fibers of arrangements with multiplicities 
\cite{cds} and for higher degree homology groups \cite{ds-multinets} have 
been known.) 

In \cite{y-dou}, the study of double coverings of 
the complement to affine hyperplane arrangements plays a crucial role. 
The key idea is to use the transfer long exact sequence for double coverings 
and $\Z_2$ cohomology groups \cite{hatcher, hau} 
in connection with the Aomoto complex. Recently \cite{suc-boc} 
applies the idea to more general setting and problems. 

In this paper, we will focus on the two peculiar properties (i) and (ii) that 
the icosidodecahedral arrangement possesses. They seem to be 
independent. However, the main result of  this paper shows that they 
are actually equivalent for double coverings. 

In the next \S \ref{sec:main}, we will introduce several invariants related to 
double coverings of CW-complexes, and present the main theorem. 
Then in \S \ref{sec:ex}, we recall the icosidedocahedral arrangement. 
We will also exhibit a simpler example of $10$ lines for which the integral 
homology of a double covering has a $2$-torsion.

\section{Main result}

\label{sec:main}

Let $X$ be a connected CW-complex. 
Let $\omega\in H^1(X, \Z_2)$. For simplicity, we assume $\omega\neq 0$. 
Since 
$H^1(X, \Z_2)\simeq\Hom(H_1(X, \Z), \Z_2)\simeq \Hom(\pi_1(X), \Z_2)$, 
$\omega$ determines a homomorphism $\pi_1(X)\to\Z_2\simeq \{\pm1\}$. 
Since $\omega\neq 0$, 
the group homomorphism $\pi_1(X)\to\{\pm1\}$ is surjective. Then the 
$\Ker(\pi_1(X)\to\{\pm1\})$ is a subgroup of $\pi_1(X)$ of index $2$, 
which determines 
the associated double covering $p_\omega:X^\omega\to X$. 
The group homomorphism $\pi_1(X)\to\{\pm1\}=\Z^\times$ also 
induces a local system of rank one over $\Z$ which we denote 
by $\LL_\omega$. 

\begin{definition} 
(The rank of $\omega$-twisted local system homology group.) 
Denote the rank of the local system homology group with coefficients in 
$\LL_\omega$ by 
\begin{equation}
\rho_k(\omega):=\rank_\Z H_k(X, \LL_\omega)=\dim_\C 
H_k(X, \LL_\omega\otimes_\Z\C), 
\end{equation}
\end{definition}

From now on, we assume that the element 
$\omega\in H^1(X, \Z_2)$ satisfies $\omega\wedge\omega=0$. Then, the multiplication 
map 
\begin{equation}
\omega\wedge:H^\bullet(X, \Z_2)\to H^{\bullet+1}(X, \Z_2)
\end{equation}
induces a cochain complex which we call the mod $2$ 
Aomoto complex $(H^\bullet(X, \Z_2), \omega\wedge)$. 

\begin{definition} 
(The rank of the cohomology of the mod $2$ Aomoto complex.) 
Denote the rank of the cohomology of the Aomoto complex by 
\begin{equation}
\alpha_k(\omega):=\rank_{\Z_2}H^k(H^\bullet(X, \Z_2), \omega\wedge). 
\end{equation}
\end{definition}
These invariants are related to integral and mod $2$ Betti numbers of 
the double covering $X^\omega$. By Leray spectral sequence, the complex 
cohomology group of $X^\omega$ decomposes into a direct sum as 
$H^k(X^\omega, \C)\simeq H^k(X, \C)\oplus H^k(X, \LL_\omega\otimes_\Z\C)$. 
Therefore, we have 
\begin{equation}
\label{leray}
b_k(X^\omega)=b_k(X)+\rho_k(\omega). 
\end{equation}

In \cite{ps-sp}, Papadima and Suciu proved the inequality 
\begin{equation}
\label{psineq}
\rho_k(\omega)\leq\alpha_k(\omega). 
\end{equation}
Therefore, the $k$-th Betti number of the double covering 
$X^\omega$ is bounded by the sum 
\begin{equation}
b_k(X^\omega)\leq b_k(X)+\alpha_k(\omega). 
\end{equation}
In \cite{PS-modular}, they conjectured the equality 
``$\rho_k(\omega)=\alpha_k(\omega)$'' holds for $X=M(\A)$ the 
complement to a complex hyperplane arrangement and the covering 
$F\longrightarrow M(\A)$ which is the Milnor fiber of the cone of $\A$. 

Recently, a counterexample to the equality was found. 
Let $\decID=\{H_1, \dots, H_{15}\}$ be as in Figure \ref{fig:ID} 
(see Example \ref{ex:id} for details). Let 
$e_1, \dots, e_{15}\in H^1(M(\decID), \Z_2)$ be the dual basis 
to the basis of $H_1(M(\decID), \Z_2)$ determined by the meridians 
of each line. Consider $\omega=e_1+e_2+\cdots + e_{15}$. Then, 
$\rho_1(\omega)=0$ and $\alpha_1(\omega)=1$ (\cite{y-dou}). 
Thus we have the strict inequality 
\begin{equation}
\label{strict}
\rho_1(\omega)<\alpha_1(\omega). 
\end{equation}
In \cite{y-dou} it was also proved that $H_1(M(\decID)^\omega, \Z)$ 
has a $2$-torsion. 

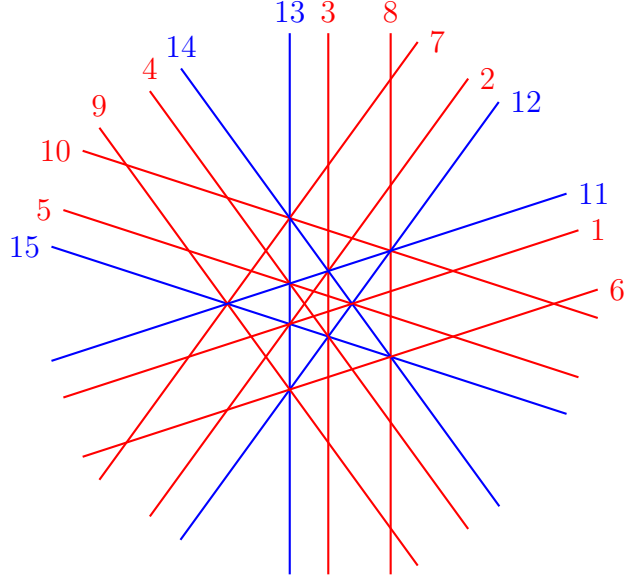
\begin{figure}[htbp]
\centering
\begin{tikzpicture}[scale=1.2]

\draw[thick, red] (0.809,0)++(0,-3) -- +(90:6) node[above, red]{$8$};
\draw[thick, blue] (-0.309,0)++(0,-3) -- +(90:6) node[above, blue]{$13$};
\draw[thick, red] (0.118,0)++(0,-3) -- +(90:6) node[above, red]{$3$};

\draw[thick, red] (72:0.809)++(162:3)  node[left, red]{$10$}-- +(342:6);
\draw[thick, blue] (252:0.309)++(162:3)  node[left, blue]{$15$}-- +(342:6);
\draw[thick, red] (72:0.118)++(162:3)  node[left, red]{$5$}-- +(342:6);

\draw[thick, red] (144:0.809)++(234:3) -- +(54:6) node[right, red]{$7$};
\draw[thick, blue] (322:0.309)++(234:3) -- +(54:6) node[right, blue]{$12$};
\draw[thick, red] (144:0.118)++(234:3) -- +(54:6) node[right, red]{$2$};

\draw[thick, red] (216:0.809)++(126:3)  node[above, red]{$9$}-- +(306:6);
\draw[thick, blue] (36:0.309)++(126:3)  node[above, blue]{$14$}-- +(306:6);
\draw[thick, red] (216:0.118)++(126:3)  node[above, red]{$4$}-- +(306:6);

\draw[thick, red] (288:0.809)++(198:3) -- +(18:6) node[right, red]{$6$};
\draw[thick, blue] (108:0.309)++(198:3) -- +(18:6) node[right, blue]{$11$};
\draw[thick, red] (288:0.118)++(198:3) -- +(18:6) node[right, red]{$1$};

\end{tikzpicture}
\caption{(A deconing of the icosidodecahedral arrangement) 
$\decID=\{H_1, \dots, H_{15}\}$} 
\label{fig:ID}
\end{figure}

The main result of this paper is a refinement of Papadima-Suciu's 
inequality (\ref{psineq}). 
Actually, the gap between 
$\rho_k(\omega)$ and $\alpha_k(\omega)$ can be precisely measured 
by $2$-torsions. To state the main result, we need the following. 

\begin{definition}
Let 
\[
H_k(X^\omega, \Z)[2]:=\{\alpha\in H_k(X^\omega, \Z)\mid 2\alpha=0\}
\]
be the $2$-torsion part of the $k$-th homology group of $X^\omega$. 
Note that the abelian group 
$H_k(X^\omega, \Z)[2]$ can be considered as a vector space 
over $\Z_2$. We denote its rank by 
\begin{equation}
\tau_k(X^\omega):=\rank_{\Z_2}H_k(X^\omega, \Z)[2]. 
\end{equation}
\end{definition}
Note that $\tau_k(X^\omega)\neq 0$ if and only if $H_k(X^\omega, \Z)$ 
has a non-trivial $2$-torsion element. More precisely, 
$\tau_k(X^\omega)$ is the number of even order summands 
when we express the torsion part of $H_k(X^\omega, \Z)$ as a direct 
sum of finite cyclic groups.  

\begin{theorem}
\label{thm:main}
Let $X$ be a connected CW-complex and $\omega\in H^1(X, \Z_2)$ with 
$\omega\wedge\omega=0$. Then, 
\begin{equation}
\label{main}
\alpha_k(\omega)=\rho_k(\omega)+\tau_k(X^\omega)+\tau_{k-1}(X^\omega).
\end{equation}
In particular, the equality $\alpha_k(\omega)=\rho_k(\omega)$ holds if and only if 
$H_k(X^\omega, \Z)$ and 
$H_{k-1}(X^\omega, \Z)$ do not have non-trivial $2$-torsion elements. 
\end{theorem}
\begin{proof}
We compute the rank of the mod $2$ cohomology group 
$\rank_{\Z_2}H^k(X^\omega, \Z_2)$ in two ways. First we apply 
the Universal coefficient theorem for cohomology 
(see e.g. \cite[Theorem 3.2]{hatcher}). 
We have 
\begin{equation}
H^k(X^\omega, \Z_2)\simeq\Hom(H_k(X^\omega, \Z), \Z_2)\oplus
\Ext(H_{k-1}(X^\omega, \Z), \Z_2). 
\end{equation}
Using the equality \ref{leray} and definitions, 
it is easily seen that the $\Z_2$-rank of the right-hand side is 
equal to $b_k(X)+\rho_k(\omega)+\tau_k(X^\omega)+\tau_{k-1}(X^\omega)$. 
Secondly, using the formula \cite[Theorem 3.7]{y-dou}, we obtain the following. 
\begin{equation}
\rank_{\Z_2}H^k(X^\omega, \Z_2)=b_k(X)+\alpha_k(\omega). 
\end{equation}
Thus we have the formula (\ref{main}). 
\end{proof}

As a special case of $k=1$, we have the following. 
\begin{corollary}
\label{cor:main}
\begin{equation}
\label{cor}
\alpha_1(\omega)=\rho_1(\omega)+\tau_1(X^\omega). 
\end{equation}
Thus the strict inequality $\alpha_1(\omega)>\rho_1(\omega)$ holds 
if and only if $H_1(X^\omega, \Z)$ has non-trivial $2$-torsion elements. 
\end{corollary}
\begin{proof}
Since $H_0(X^\omega, \Z)$ is torsion free, we have $\tau_0(X^\omega)=0$. 
\end{proof}

Combining (\ref{leray}) and 
Theorem \ref{thm:main}, we have the following. 
\begin{corollary}
If $\omega\wedge\omega=0$, the Betti number of the double cover is 
\begin{equation}
b_k(X^\omega)=b_k(X)+\alpha_k(\omega)-
\tau_k(X^\omega)-\tau_{k-1}(X^\omega). 
\end{equation}
In particular, 
$b_1(X^\omega)=b_1(X)+\alpha_1(\omega)-\tau_1(X^\omega)$. 
\end{corollary}

\begin{remark}
\label{rem:liu}
Let us consider the eigenspace decomposition (\ref{eq:eigen}) 
of the Milnor fiber $F_\A$ for 
a central arrangement $\A$ with $|\A|$ even. 
In this case, the $(-1)$-eigenspace 
$H_k(F_\A, \C)_{-1}$ appears as a direct summand of $H^k(F_\A, \C)$. 
Recall that $H_k(F_\A, \C)_{-1}$ is isomorphic to the local system cohomology 
group $H^k(M(\A), \LL_\omega\otimes\C)$ with $\omega=\sum_i e_i$. 
Therefore, 
\begin{equation}
\dim H_k(F_\A, \C)_{-1}=\rho_k(\omega)=
\alpha_k(\omega)-\tau_k(X^\omega)-\tau_{k-1}(X^\omega), 
\end{equation}
where $X=M(d\A)$ is the complement of the deconing $d\A$ 
(or equivalently, projectivized complement $X=M(\A)/\C^\times$). 
In particular, 
$\dim H_1(F_\A, \C)_{-1}=\alpha_1(\omega)-\tau_1(X^\omega)$. 
Thus, to compute the dimension of the $(-1)$-eigenspace 
$\dim H_k(F_\A, \C)_{-1}$, 
the rank of $2$-torsion part $\tau_k(X^\omega)$ 
of the homology of the double covering $H_k (X^\omega, \Z)$ 
is unavoidable. A combinatorial description of 
$\dim H_k(F_\A, \C)_{-1}$ must involve a combinatorial description of 
$\tau_k(X^\omega)$. 
\end{remark}

\section{Examples related to the icosidodecahedral arrangement}
\label{sec:ex}

\begin{example} 
\label{ex:id}
((The deconing of the) icosidodecahedral arrangement \cite{y-dou}) 
Let $\decID=\{H_1, \dots, H_{15}\}$ be the arrangement of affine 
$15$ lines as in Figure \ref{fig:ID}. 
The first homology $H_1(M(\decID), \Z)$ of the complement 
$M(\decID)=\C^2\smallsetminus\bigcup_{i=1}^{15} H_i\otimes\C$ is 
generated by meridians of each $H_i$. Let $e_1, \dots , e_{15}
\in H^1(M(\decID), \Z_2)$ be the dual basis to meridians. 

(1) Let $\omega=e_1+e_2+\dots+e_{15}$ be the sum of all $e_i$'s. 
Let $\eta=e_{11}+\cdots+e_{15}$. Then $\omega\wedge\eta=0$. 
We can show that the cohomology of the Aomoto complex 
$H^1(H^\bullet(M(\decID), \Z_2), \omega\wedge)$ is rank one 
generated by $\eta$. Thus we have 
$\alpha_1(\omega)=1$. We also have $\rho_1(\omega)=0$ (see 
\cite[Theorem 4.3]{y-dou} for details). Hence 
$\tau_1(X^\omega)=1$ and $H_1(X^\omega, \Z)$ has a non-trivial 
$2$-torsion element. 

(2) There are many other choices of $\omega$ such that 
$\alpha_1(\omega)=1$ and $\rho_1(\omega)=0$. For example, 
$\omega=e_5+e_6+e_9+e_{10}+e_{12}+e_{13}+e_{15}$. Hence, there 
exist many double coverings of $M(\decID)^\omega$ of 
$M(\decID)$ with torsion in $H_1(M(\decID)^\omega, \Z)$. 
As far as the 
authors checked by computer, for every 
$\omega$ with $\alpha_1(\omega)=1$ and $\rho_1(\omega)=0$, 
the first homology of the double covering and 
the associated local system homology group are 
\begin{equation}
\begin{split}
H_1(M(\decID)^\omega, \Z)&\simeq\Z^{\oplus 15}\oplus \Z_2, \\
H_1(M(\decID), \LL_{\omega})&\simeq\Z_2^{\oplus 13}\oplus \Z_4. 
\end{split}
\end{equation}

(3) Another interesting element is $\omega=e_{11}+e_{12}+e_{13}+e_{14}+e_{15}$. 
Then the rank of the cohomology of the Aomoto complex is $\alpha_1(\omega)=6$. 
The rank of the local system cohomology is also 
$\rho_1(\omega)=6$. Hence $\tau_1(M(\decID)^\omega)=0$.  
\end{example}

The above example shows that several double coverings of 
$M(\decID)$ have non-trivial $2$-torsions in the integral homology groups 
of double coverings. The above example $\decID$ is not the 
smallest example with non-trivial $2$-torsions. 
One of subarrangements of $\decID$ also have $2$-torsions as follows. 

\begin{example} 
\label{ex:ds}
(Double star arrangement) 
Let $\DP$ be the subarrangement of $\decID$ consisting of 
$10$ lines $\{H_6, \dots, H_{10}, H_{11}, \dots, H_{15}\}$ 
(Figure \ref{fig:DP}). 

\begin{figure}[htbp]
\centering
\begin{tikzpicture}[scale=1.2]

\draw[thick, red] (0.809,0)++(0,-3) -- +(90:6) node[above, red]{$8$};
\draw[thick, blue] (-0.309,0)++(0,-3) -- +(90:6) node[above, blue]{$13$};
%\draw[thick, red] (0.118,0)++(0,-3) -- +(90:6) node[above, red]{$3$};

\draw[thick, red] (72:0.809)++(162:3)  node[left, red]{$10$}-- +(342:6);
\draw[thick, blue] (252:0.309)++(162:3)  node[left, blue]{$15$}-- +(342:6);
%\draw[thick, red] (72:0.118)++(162:3)  node[left, red]{$5$}-- +(342:6);

\draw[thick, red] (144:0.809)++(234:3) -- +(54:6) node[right, red]{$7$};
\draw[thick, blue] (322:0.309)++(234:3) -- +(54:6) node[right, blue]{$12$};
%\draw[thick, red] (144:0.118)++(234:3) -- +(54:6) node[right, red]{$2$};

\draw[thick, red] (216:0.809)++(126:3)  node[above, red]{$9$}-- +(306:6);
\draw[thick, blue] (36:0.309)++(126:3)  node[above, blue]{$14$}-- +(306:6);
%\draw[thick, red] (216:0.118)++(126:3)  node[above, red]{$4$}-- +(306:6);

\draw[thick, red] (288:0.809)++(198:3) -- +(18:6) node[right, red]{$6$};
\draw[thick, blue] (108:0.309)++(198:3) -- +(18:6) node[right, blue]{$11$};
%\draw[thick, red] (288:0.118)++(198:3) -- +(18:6) node[right, red]{$1$};

\end{tikzpicture}
\caption{(The double star arrangement) 
$\DP=\{H_6, \cdots, H_{15}\}$} 
\label{fig:DP}
\end{figure}

Let $\omega=e_6+e_7+\dots+e_{15}$ be the sum of all $e_i$'s. 
As Example \ref{ex:id}, $\alpha_1(\omega)=1$ and $\rho_1(\omega)=0$. 
Hence $\tau_1(X^\omega)=1$ and $H_1(X^\omega, \Z)$ has a non-trivial 
$2$-torsion element. As far as the authors checked by computer, 
it is the unique $\omega$ with $\tau_1(X^\omega)=1$. 
The first homology of the double covering and 
the associated local system homology group are 
\begin{equation}
\begin{split}
H_1(M(\DP)^\omega, \Z)&\simeq
\Z^{\oplus 10}\oplus \Z_2,\\
H_1(M(\DP), \LL_\omega)&\simeq
\Z_2^{\oplus 8}\oplus \Z_4. 
\end{split}
\end{equation}
\end{example}

The computations in Example \ref{ex:id} and Example \ref{ex:ds} 
suggest that $2$-torsions are closely related to the local system 
(co)homology groups with rank one local systems over $\Z$, which 
have been recently studied in \cite{sug} 
for complexified real arrangements. 
We pose the following conjecture. 

\begin{conjecture}
Let $\A$ be a complex hyperplane arrangement. 
Let $\omega\in H^1(M(\A), \Z_2)$. The integral homology 
$H_1(M(\A)^\omega, \Z)$ has a non-trivial $2$-torsion if and only if 
the local system homology $H_1(M(\A), \LL_\omega)$ has a $4$-torsion 
($\Z_4$-summand). 
\end{conjecture}

\emph{Acknowledgements}. 
We thank Philip Tosteson for his inspiring comments 
during the conference 
Hyperplane Arrangements and Singularities (December 2019, Tokyo). 
We also thank Ye Liu for comments (Remark \ref{rem:liu}) 
on the previous version of the paper. 
This work was partially supported by JSPS KAKENHI 
Grant Numbers JP18H01115, JP23H00081 and 22J20470.

\end{document}